\documentclass{commat}

\newcommand{\A}{\mathcal{A}}
\newcommand{\cd}[1]{\risS{-9}{#1}{}{24}{15}{15}}
\newcommand{\cdO}{\rb{1pt}{\chd{cd1ch4}}}
\newcommand{\cdTF}{\rb{1pt}{\chd{cd34ch4}}}
\newcommand{\cdTO}{\rb{1pt}{\chd{cd31ch4}}}
\newcommand{\cdTT}{\rb{1pt}{\chd{cd33ch4}}}
\newcommand{\cdTV}{\rb{1pt}{\chd{cd35ch4}}}
\newcommand{\cdTW}{\rb{1pt}{\chd{cd32ch4}}}
\newcommand{\cdWO}{\rb{1pt}{\cd{cd21ch4}}}
\newcommand{\cdWW}{\rb{1pt}{\cd{cd22ch4}}}
\newcommand{\chd}[1]{\risS{-10}{#1}{}{25}{20}{15}}
\newcommand{\cld}[1]{\risS{-9.5}{#1}{}{23}{20}{15}}
\newcommand{\ig}{\includegraphics}
\newcommand{\rb}{\raisebox}
\newcommand\risS[6]{\rb{#1pt}[#5pt][#6pt]{\begin{picture}(#4,15)(0,0)
  \put(0,0){\ig[width=#4pt]{#2.eps}} #3
     \end{picture}}}

\title{Partial-dual genus polynomial as a weight system}

\author{Sergei Chmutov}

\authorinfo{The Ohio State University at Mansfield, USA.}{chmutov.1@osu.edu}

\abstract{
    We prove that the partial-dual genus polynomial considered as a function on chord diagrams satisfies the four-term relation. Thus it is a weight system from the theory of Vassiliev knot invariants.
    }



\VOLUME{31}
\YEAR{2023}
\NUMBER{3}
\firstpage{113}
\DOI{https://doi.org/10.46298/cm.11709}

\begin{document}

\begin{flushright}
\includegraphics[width=220pt]{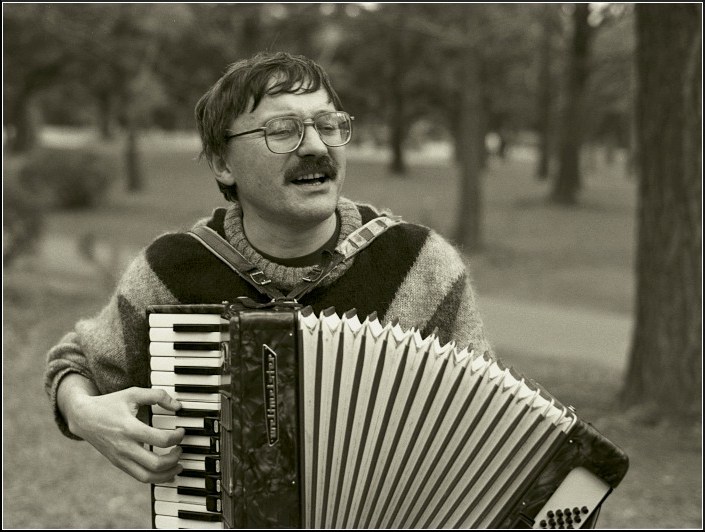}

{\it To the memory of my friend Sergei Duzhin}
\end{flushright}

\section{Introduction} \label{s:intro}

The Partial duality of ribbon graphs  relative to a subset of its edges was introduced in \cite{Ch}. This operation often changes the genus of the ribbon graph. The classical Euler–Poincar\'e duality is the partial duality relative to all the edges. For an excellent exposition see \cite{EMM}. J.~L.~Gross, T.~Mansour, and T.~W.~Tucker \cite{GMT} introduced the partial-dual genus polynomial as a genus generating function for all partial duals of a given ribbon graph. 

Chord diagrams are the main combinatorial object of the theory of Vassiliev knot invariants (see, for example , \cite{CDM}). We can consider a chord diagram as an oriented ribbon graph with a singe vertex given by the circle of the chord diagram and ribbon-edges corresponding to its chords. The four-term relation for functions on chord diagrams is the key relation in the theory of Vassiliev knot invariants. The functions on chord diagram satisfying it are called {\it weight systems}.

Here we prove that the partial-dual genus polynomial is a weight system.

In Sec.\ref{s:rg} we recall the basic notions and properties of ribbon graphs and partial duality. In Sec.\ref{s:cd} we recall the basic facts about chord diagrams.
We formulate and prove our main result in Sec.\ref{s:main}. A few open problems for further research are listed in Sec.\ref{s:op}.

I am grateful to Zhiyun Cheng and Sergei Lando for their  very interesting 
discussions and the anonymous referee's  valuable comments.

\section{Ribbon graphs} \label{s:rg}
\begin{definition}\rm
A {\it\bfseries ribbon graph} is a surface with boundary decomposed into a number of closed topological discs of two types, 
{\it\bfseries vertex-discs} and {\it\bfseries edge-ribbons}, satisfying the conditions: the discs of the same type are pairwise disjoint; the vertex-discs and the edge-ribbons intersect by disjoint line segments, each such line segment lies on the boundary of precisely one vertex and precisely one edge  and every edge contains exactly two such line segments, which are not adjacent.
\end{definition}
Ribbon graphs are considered up to a homeomorphism preserving the decomposition. We think about edge-ribbons as topological rectangles attached to the vertices along two opposite sides.

Here are a few examples.
$$\mbox{(a)}
\risS{-35}{rg-ex1}{}{45}{15}{40}\quad
  \mbox{(b)}
\risS{-30}{rg-ex2}{}{45}{0}{0}\quad
  \mbox{(c)}
\risS{-40}{basket}{}{50}{0}{0}\quad
  \mbox{(d)}
\risS{-30}{rg-ex3}{}{180}{0}{0}
$$
A ribbon graph can be regarded as a regular neighborhood of a graph cellularly embedded into a closed surface. Then the neighborhoods of vertices and edges represent vertex-discs and edge-ribbons respectively.
Obtaining a closed surface into which the core graph of a ribbon graph is embedded can be done by gluing a disc, called a {\it\bfseries face-disc}, to each boundary component.
Here is how our examples (b) and (c) look as graphs on surfaces.
$$\rb{20pt}{(b)}
\risS{-15}{rg-ex2}{}{45}{0}{0}\ \risS{-1}{totorl}{}{20}{0}{0}\ 
\risS{-20}{rg-ex2-torus}{}{85}{0}{0}\quad
\rb{20pt}{(c)}\risS{-25}{basket}{}{50}{35}{25}=
\risS{-20}{gr34}{}{70}{0}{0}\ \risS{-1}{totorl}{}{20}{0}{0}\ 
\risS{-20}{basket-torus}{}{85}{0}{0}
$$
In this paper we will deal with oriented ribbon graphs. Example (d) above is a non-orientable ribbon graph. 

The classical classification theorem says that the surface of an 
(orientable) ribbon graph is homeomorphic to a surface of genus $g$ with a number of holes corresponding to the boundary components of the ribbon graph. This is the number of face-discs we need to glue in to get a closed surface. Such a homeomorphism forgets the decomposition of the ribbon graph into vertex-discs and edge-ribbons. The genus $g$ can be determined using the Euler formula $2-2g=v-e+f$, where $v,e,f$ are the numbers of 
vertex-discs, edge-ribbons, and face-discs respectively. The ribbon graphs in examples (a), (b), and (c) are all of genus 1.

\begin{lemma}[Sliding Lemma]
Sliding one edge-ribbon of a ribbon graph along another one results in a homeomorphic surface. Thus they have equal genera.
\end{lemma}
The proof is obvious from the picture\qquad\quad
$\risS{-40}{slide-1}{}{35}{10}{0}\quad 
   \risS{-18}{totor}{}{30}{0}{0}\quad 
\risS{-40}{slide-2}{}{35}{0}{0}\quad 
   \risS{-18}{totor}{}{30}{0}{0}\quad 
\risS{-40}{slide-3}{}{35}{0}{0}$

Sliding the left loop in example (c)\\ along the middle loop gives a different\\
graph on a torus:

$\qquad
\risS{-20}{gr34-sl-1}{}{70}{35}{30}\ \risS{-1}{totor}{}{30}{0}{0}\ 
\risS{-20}{gr34-sl-2}{}{60}{0}{0}\ \risS{-3}{totorl}{}{30}{0}{0}\ 
\risS{-20}{ba-sl-torus}{}{80}{0}{0}$

\subsection{Partial duality.} \label{ss:pd}\rm
For a ribbon graph $G$ and a subset of the edge-ribbons $A\subseteq E(G)$, the {\it\bfseries partial dual}, $G^A$ of $G$ relative to $A$ is a ribbon graph constructed as follows. The vertex-discs of $G^A$ are bounded by the connected components of the boundary of the spanning subgraph of $G$ containing all the vertices of $G$ and only the edges from $A$. The edge-ribbons of $G^A$ are in
one-to-one correspondence with those of $G$. 
The edge-ribbons of
$E(G)\setminus A$ are attached to these new vertices exactly at the same places as in $G$. The edge-ribbons from $A$ become parts of the new vertex-discs now. 
For $e\in A$ we take a copy of $e$, $e'$, and attach it to the new vertex-discs in the following way. The rectangle representing $e$ intersects 
with the vertex-discs of $G$  at a pair of opposite sides.
The same rectangle intersects the boundary of the spanning subgraph,
that is, of
the new vertex-discs, along the arcs of the other pair of its opposite sides. We attach $e'$ to these arcs by
this second pair. The copies of the first pair of sides in $e'$ become the arcs of the boundary of $G^A$.

Partial duality relative to a set of edges can be done step by step one edge at a time. The partial duality relative to one edge can be illustrated as follows.
$$G =\ 
  \risS{-8}{dus1}{\put(34,-2){\mbox{$e$}}}{70}{0}{0}\quad 
   \risS{-4}{totor}{}{30}{0}{0}\quad 
  \risS{-35}{dus3}{\put(-4,13){\mbox{$e'$}}}{110}{45}{40}\ =\ 
  \risS{-35}{dus4}{\put(47,65){\mbox{$e'$}}}{70}{10}{0}\ =\ G^{\{e\}}
$$
Here the boxes with dashed arcs mean that there might be other edges attached to these vertices. Partial duality relative to a loop acts as on this figure backwards, from right to left.

Here is the partial dual of example (c) relative to the middle loop.
$$G\ = \  \risS{-20}{gr34}{\put(50,49){\mbox{$e$}}}{70}{30}{25}\qquad 
\risS{-1}{totor}{}{30}{0}{0}\qquad 
\risS{-18}{gr34-pd}{\put(40,17){\mbox{$e'$}}}{70}{0}{0}\ =\ G^{\{e\}}\ ,
$$
It is a graph cellularly embedded into a sphere.

The concept of partial duality can be generalized to hypermaps \cite{CVT22}
where the edges are represented not just by topological rectangles but rather by arbitrary polygons with an even number of sides, and every other side attached to a vertex.

\subsection{Partial-dual genus polynomial.} \label{ss:pdp}\rm
In \cite{GMT} J.~L.~Gross, T.~Mansour, and T.~W.~Tucker introduced the 
{\it\bfseries partial-dual genus polynomial} ${\ }^\partial\Gamma_G(z)$ 
as the generating function for the number of all partial duals of $G$ of the given genus. 
$${\ }^\partial\Gamma_G(z):=\sum_{A\subseteq E(G)} z^{g(G^A)}$$
For example, the partial-dual genus polynomial of the ribbon graph 
$G_{\mbox{\scriptsize (b)}}$ in example (b) is
${\ }^\partial\Gamma_{G_{\mbox{\scriptsize (b)}}}(z) = 2+2z$.
The sum of the coefficients of the polynomial is always $2^e$ because that is the total number of subsets of edges of $G$ and therefore the total number of its partial duals.

The partial-dual genus polynomial was studied in numerous papers 
\cite{GMT21,GMT21e,CVT21,YaJi21,YaJi22,YaJi22b,YaJi22s} and its generalization to delta-matroids in \cite{YaJi22dm,Yus22}.

\section{Chord diagrams} \label{s:cd}
The partial dual of a connected ribbon graph relative to a spanning tree is a one-vertex ribbon graph. In \cite{GMT} such graphs are called bouquets. We will encode them by chord diagrams which are the main combinatorial object of the theory of Vassiliev knot invariants \cite{CDM}.
\begin{definition}\rm \label{d:cd}
A {\it\bfseries chord diagram} of order $n$ (or degree $n$) is an oriented
circle with a distinguished set of $n$ disjoint pairs of distinct
points, considered up to orientation preserving homeomorphisms of
the circle. The set of all chord diagrams of order $n$ will be
denoted by $\mathbf{A}_n$.
\end{definition}
We shall usually omit the orientation of the circle in pictures of
chord diagrams, assuming that it is oriented counterclockwise. Also,  we indicate the points in a pair connecting them by a {\it\bfseries chord}.

\medskip
\noindent
{\bf Examples.} Here are the chord diagrams with  at most 3 chords.
$$\mathbf{A}_1=\Bigl\{\cdO\Bigr\},\quad
\mathbf{A}_2=\Bigl\{\cdWO,\cdWW\Bigr\},\quad
\mathbf{A}_3=\Bigl\{\cdTO,\cdTW,\cdTT,\cdTF,\cdTV\Bigr\}.
$$
There are 18 chord diagrams with 4 chords:
$$\begin{array}{ccccccccc}
\cd{cd4-01}\ , & \cd{cd4-02}\ , & \cd{cd4-03}\ , & \cd{cd4-04}\ , & 
\cd{cd4-05}\ , & \cd{cd4-06}\ , &
\cd{cd4-07}\ , & \cd{cd4-08}\ , & \cd{cd4-09}\ , \\
\cd{cd4-10}\ , &  \cd{cd4-11}\ , & \cd{cd4-12}\ , &
\cd{cd4-13}\ , & \cd{cd4-14}\ , & \cd{cd4-15}\ , & \cd{cd4-16}\ , & 
\cd{cd4-17}\ , & \cd{cd4-18}\ .
\end{array}
$$

\subsection{Four-term relation for chord diagrams.} \label{ss:4T}\rm
The functions on chord diagrams satisfying the four-term relation play a key role in the theory of Vassiliev knot invariants. We can take a dual point of view considering the space spanned by all chord diagrams and impose the four-term relation on chord diagrams \cite[Sec.4.11]{CDM}:\vspace{-7pt}
\begin{equation}\tag{4T}\label{eq:4T}
\chd{4T1}\ -\ \chd{4T2}\ +\ \chd{4T4}\ -\ \chd{4T3}\ =0. 
\end{equation}
Here we assume that the diagrams in the pictures may have other
chords with endpoints on the dotted arcs, while all the endpoints of
the chords on the solid portions of the circle are explicitly shown.
For example, this means that in the first and second diagrams the
two bottom points are adjacent. The chords omitted from the pictures
should be the same in all  four diagrams. We say that a function on chord diagram satisfies the four-term relation if the alternating sum of its values on the four diagrams of \eqref{eq:4T} equals zero.

The theory of Vassiliev knot invariants deals with the vector space $\A_n$ which is the quotient of the space spanned by all chord diagrams $\mathbf{A}_n$
modulo the subspace spanned by all 4-term linear combinations.

Here are the dimensions and some bases of the spaces $\A_n$ for $n=1$, 2 and 3:

$\A_1=\bigl<\chd{cd1ch4}\bigr>$, $\dim\A_1=1$.\qquad
$\A_2=\bigl<\cld{cd21ch4},\cld{cd22ch4}\bigr>$,
$\dim\A_2=2$, since the only 4-term relation involving chord
diagrams of order 2 is trivial.\qquad
$\A_3=\bigl<\chd{cd32ch4},\chd{cd33ch4},\chd{cd34ch4}\bigr>$,
$\dim\A_3=3$, since $\mathbf{A}_3$ consists of 5 elements, and there
are 2 independent 4-term relations:\vspace{-10pt}
$$\chd{cd31ch4}\ =\ \chd{cd32ch4}\qquad\mbox{and}\qquad
\chd{cd35ch4}\ -\ 2\,\chd{cd34ch4}\ +\ \chd{cd33ch4}\ =\ 0.
$$

The {\it\bfseries multiplication} of two chord diagrams $D_1$ and $D_2$ is defined by a connected sum, that is by cutting
and gluing the two circles as shown:\vspace{-10pt}
$$\rb{-12pt}{\rotatebox{90}{\chd{cd34ch4}}}\!\!\cdot\ \chd{cd34ch4}\ =\
\rb{-3.9mm}{\ig[width=22mm]{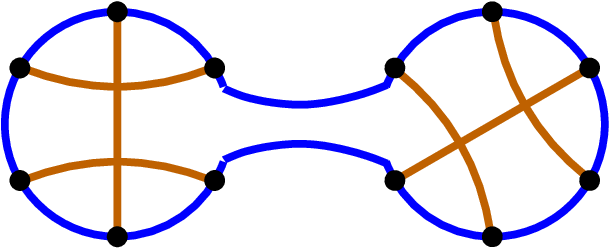}}\ =\ \chd{cd_prod3}
$$
Of course, the product of diagrams depends on the choice of the
points where the diagrams are cut, but remarkably modulo \eqref{eq:4T}, it does not. There is also a comultiplication of chord diagrams which turns the space $\sum\limits_{n=0}^\infty \A_n$ into a graded bialgebra. As such, it is isomorphic to the polynomial algebra over the primitive space. See
\cite{CDM} for details.

\subsection{Ribbon graphs from chord diagrams.} \label{ss:rg-cd}\rm
We can construct a ribbon graph from a chord diagram. Namely , we attach a vertex-disc to the circle of a chord diagram thinking about is as a hemisphere under the plane of the chord diagram. Then we 
thicken the chords to narrow edge-ribbons.
Here is an example.
$$D\ =\ \risS{-9}{cd34ch4}{}{24}{0}{0}\qquad 
\risS{-1}{totorl}{}{30}{0}{0}\qquad
\risS{-25}{basket}{}{50}{30}{25}\ =\ 
\risS{-20}{gr34}{}{70}{0}{0}\ =\ G_D
$$
Thus the ribbon graph of example (c) is associated with the fourth chord diagram in $\mathbf{A}_3$. 

The {\it\bfseries ribbon-join operation} $G_1\vee G_2$ of one-vertex  ribbon graphs $G_1$ and $G_2$ from \cite{Mo3} corresponds to the multiplication of chord diagrams.

\bigskip
It is easy to interpret sliding operations on chord diagram. Here is how it looks for example (c)
$$\risS{-9}{cd34-sl1}{}{24}{0}{20}\ \risS{-1}{totor}{}{20}{0}{0}\
\risS{-9}{cd34-sl2}{}{24}{0}{0}\ \risS{-1}{totor}{}{20}{0}{0}\ 
\risS{-9}{cd34-sl3}{}{24}{0}{0}\ =\
\risS{-9}{cd34-sl4}{}{24}{0}{0}\ \risS{-1}{totor}{}{20}{0}{0}\
\risS{-9}{cd34-sl5}{}{24}{0}{0}\ =\ 
\risS{-9}{cd35ch4}{}{24}{0}{0}\ .
$$
Using sliding operation on chord diagrams it is easy to reformulate and prove the classical classification theorem for surfaces \cite{Ch20}.
\begin{theorem}
Any (orientable) surface of genus $g$ with $k+1$ boundary components is homeomorphic to the surface associated with a caravan of $k$ ``one-humped camels" and $g$ ``two-humped camels".
$$\risS{-12}{caravan}{\put(18,-8){$k$}\put(77,-8){$g$}}{114}{8}{25}\ ,
$$
which is a product of $k$ chord diagrams with a single chord and $g$ copies of the second chord diagram from $\mathbf{A}_2$ in Sec.\ref{d:cd}.
\end{theorem}
The non-orientable surfaces can be treated in a similar way, see \cite{Ch20}.

\subsection{Partial duality in terms of chord diagrams.} \label{ss:pd-cd}\rm
In order to describe partial duality in the language of chord diagrams we need to encode multi-vertex ribbon graphs by chord diagrams as well. This can be done using chord diagrams on several circles 
(see \cite[Sec.5.10.1]{CDM} for an even more general concept of Jacobi diagrams on any tangle). We will follow the convention that all circles of such a diagram are depicted outside of each other and that both ends of a chord are attached either from the inside of the corresponding circles or from the outside.
Here is the chord diagram on two circles corresponding to the two-vertex ribbon graph of example (a).
$$\risS{-25}{cd-rg-ex1}{}{30}{0}{25}$$

To perform the partial duality relative to a single edge-ribbon
(that is a chord), we need to break the corresponding circles at the end of the chord and reconnect them by two arcs parallel to the chord, and then place a new chord connecting these two arcs.
In this process, we have to respect the orientations of the circles of the chord diagram so that the new created circles will be still oriented counterclockwise. We illustrate this process by this picture:
$$\risS{-15}{pd-cd}{}{30}{20}{25}\quad\risS{-1}{totor}{}{20}{0}{0}\quad
\risS{-16}{pd-cd1}{}{32}{0}{0}
$$
Here is an example corresponding to the one of Sec.\ref{ss:pd}, expressed in terms of chord diagrams where we perform the partial duality relative to a horizontal chord.
$$
\risS{-11}{cd34ch4}{}{32}{30}{20}\quad\risS{-1}{totor}{}{20}{0}{0}\quad
\risS{-11}{cd34-pd}{}{30}{0}{0}\ =\ \risS{-25}{cd34-pd1}{}{35}{0}{0}
$$

\medskip
\noindent Naturally the partial-dual genus polynomial is multiplicative relative\\ to the product of chord diagrams.

Now we can easily calculate the partial-dual genus polynomial for 
chord diagrams with a small number of chords:
$${\ }^\partial\Gamma_{\!\!\risS{-15}{cd1ch4}{}{19}{0}{19}\!\!}(z) = 2; \qquad 
{\ }^\partial\Gamma_{\!\!\risS{-15}{cd21ch4}{}{19}{0}{0}\!\!}(z) = 4, \qquad
{\ }^\partial\Gamma_{\!\!\risS{-15}{cd22ch4}{}{19}{0}{0}\!\!}(z) = 2+2z;
$$
$${\ }^\partial\Gamma_{\!\!\risS{-15}{cd31ch4}{}{19}{0}{19}\!\!}(z) = 8, \quad 
{\ }^\partial\Gamma_{\!\!\risS{-15}{cd32ch4}{}{19}{0}{0}\!\!}(z) = 8, \quad
{\ }^\partial\Gamma_{\!\!\risS{-15}{cd33ch4}{}{19}{0}{0}\!\!}(z) = 4+4z, \quad
{\ }^\partial\Gamma_{\!\!\risS{-15}{cd34ch4}{}{19}{0}{0}\!\!}(z) = 2+6z, \quad 
{\ }^\partial\Gamma_{\!\!\risS{-15}{cd35ch4}{}{19}{0}{0}\!\!}(z) = 8z.
$$

\bigskip
For the $18$ chord diagrams with $4$ chords we list the  partial-dual genus polynomial in the following table. We name these diagrams 
$d^4_1,\dots,d^4_{18}$ following \cite[Sec.4.4.1]{CDM}. We also provide the 
{\it interlace sequence} from \cite{YaJi21} used for enumeration of bouquet ribbon graphs. Finally, modulo the four-term relation in $\A_4$ there are only $6$ independent chord diagrams. We can choose the set
$\{d^4_3,d^4_6,d^4_7,d^4_{15},d^4_{17},d^4_{18}\}$ as a basis in this vector space. We highlight these elements and give the expression of all other chord diagrams in terms of this basis.

\begin{table}
    \centering
    \resizebox{\textwidth}{!}{
        \begin{tabular}{|c||l||c||l|}
        \hline
        CD&\parbox{2.4in}{name = interlace sequence;\\the partial-dual genus polynomial;\\ relation modulo  \eqref{eq:4T}.}&
        CD&\parbox{2.4in}{name = interlace sequence;\\the partial-dual genus polynomial;\\ relation modulo  \eqref{eq:4T}.}\\ \hline\hline
        \risS{-10}{cd4-01}{}{25}{20}{20}&
        \parbox{2.2in}{\medskip $d^4_1=(3,3,3,3)$; \\[5pt]
        ${\ }^\partial\Gamma_{d^4_1}(z) = 8z+8z^2$;\\[5pt] 
        $d^4_1=d^4_3+2d^4_6-d^4_7-2d^4_{15}+d^4_{17}$.}
             &%
        \chd{cd4-02} &
        \parbox{2.4in}{\medskip $d^4_2=(2,2,2,2)$; \\[5pt]
        ${\ }^\partial\Gamma_{d^4_2}(z) = 2+10z+4z^2$;\\[5pt] 
        $d^4_2=d^4_3-d^4_6+d^4_7$.}
             \\[25pt] \hline
        
        \colorbox{yellow}{\risS{-10}{cd4-03}{}{25}{20}{20}}&
        \parbox{2.2in}{\medskip \colorbox{yellow}{$d^4_3=(2,2,3,3)$};\\[5pt] 
        ${\ }^\partial\Gamma_{d^4_3}(z) = 12z+4z^2$.}
             &%
        \chd{cd4-04} &
        \parbox{2.4in}{\medskip $d^4_4=(1,1,1,3)$;\\[5pt] 
        ${\ }^\partial\Gamma_{d^4_4}(z) = 2+14z$;\\[5pt] 
        $d^4_4=d^4_6-d^4_7+d^4_{15}$.}
             \\[25pt] \hline
        
        \risS{-10}{cd4-05}{}{25}{20}{20}&
        \parbox{2.2in}{\medskip $d^4_5=(1,2,2,3)$;\\[5pt]
        ${\ }^\partial\Gamma_{d^4_5}(z) = 12z+4z^2$;\\[5pt] 
        $d^4_5=2d^4_6-d^4_7$.
        }
             &%
        \colorbox{yellow}{\risS{-10}{cd4-06}{}{25}{20}{20}} &
        \parbox{2.4in}{\medskip \colorbox{yellow}{$d^4_6=(1,1,2,2)$};\\[5pt] 
        ${\ }^\partial\Gamma_{d^4_6}(z) = 2+10z+4z^2$.}
             \\[25pt] \hline
        
        \colorbox{yellow}{\risS{-10}{cd4-07}{}{25}{20}{20}}&
        \parbox{2.2in}{\medskip \colorbox{yellow}{$d^4_7=(1,1)\vee(1,1)$};\\[5pt] 
        ${\ }^\partial\Gamma_{d^4_7}(z) = (2+2z)^2$.}
             &%
        \chd{cd4-08} &
        \parbox{2.4in}{\medskip $d^4_8=(0)\vee(0)\vee(0)\vee(0)$;\\[5pt] 
        ${\ }^\partial\Gamma_{d^4_2}(z) = 16$;\quad $d^4_8=d^4_{18}$.
        }
             \\ \hline
        
        \risS{-10}{cd4-09}{}{25}{20}{20} &
        \parbox{2.2in}{\medskip $d^4_9=(0)\vee(0)\vee(1,1)$;\\[5pt]
        ${\ }^\partial\Gamma_{d^4_9}(z) = 4(2+z)$;\quad $d^4_9=d^4_{17}$.}
             &%
        \chd{cd4-10}  &
        \parbox{2.4in}{\medskip $d^4_{10}=(0)\vee(0)\vee(1,1)$;\\[5pt]
        ${\ }^\partial\Gamma_{d^4_{10}}(z) = 4(2+z)$;\quad $d^4_{10}=d^4_{17}$.}
             \\ \hline
        		
        \risS{-10}{cd4-11}{}{25}{20}{20} &
        \parbox{2.2in}{\medskip $d^4_{11}=(0)\vee(1,1,2)$;\\[5pt]
        ${\ }^\partial\Gamma_{d^4_{11}}(z) = 2(2+6z)$;\\[5pt] $d^4_{11}=d^4_{15}$.}
             &%
        \chd{cd4-12} &
        \parbox{2.4in}{\medskip $d^4_{12}=(0)\vee(1,1,2)$;\\[5pt]
        ${\ }^\partial\Gamma_{d^4_{12}}(z) = 2(2+6z)$;\quad $d^4_{12}=d^4_{15}$.}
             \\[25pt] \hline
        		
        \risS{-10}{cd4-13}{}{25}{20}{20} &
        \parbox{2.2in}{\medskip $d^4_{13}=(0)\vee(2,2,2)$;\\[5pt]
        ${\ }^\partial\Gamma_{d^4_{13}}(z) = 16z$;\\[5pt]
        $d^4_{13}=2d^4_{15}-d^4_{17}$.}
             &%
        \chd{cd4-14}  &
        \parbox{2.4in}{\medskip $d^4_{14}=(0)\vee(0)\vee(1,1)$;\\[5pt]
        ${\ }^\partial\Gamma_{d^4_{14}}(z) = 4(2+z)$;\quad $d^4_{14}=d^4_{17}$.}
             \\[25pt] \hline
        		
        \colorbox{yellow}{\risS{-10}{cd4-15}{}{25}{20}{20}} &
        \parbox{2in}{\medskip \colorbox{yellow}{$d^4_{15}=(0)\vee(1,1,2)$};\\[5pt] 
        ${\ }^\partial\Gamma_{d^4_{11}}(z) = 2(2+6z).$}
             &%
        \chd{cd4-16}  &
        \parbox{2.4in}{\medskip $d^4_{16}=(0)\vee(0)\vee(0)\vee(0)$;\\[5pt]
        ${\ }^\partial\Gamma_{d^4_{16}}(z) = 16$;\quad $d^4_{16}=d^4_{18}$.}
             \\ \hline
        		
        \colorbox{yellow}{\risS{-10}{cd4-17}{}{25}{20}{20}} &
        \parbox{2in}{\medskip 
        \colorbox{yellow}{$d^4_{17}=(0)\vee(0)\vee(1,1)$};\\[5pt] 
        ${\ }^\partial\Gamma_{d^4_{17}}(z) = 4(2+z).$}
             &%
        \colorbox{yellow}{\risS{-10}{cd4-18}{}{25}{20}{20}} &
        \parbox{2in}{\medskip 
        \colorbox{yellow}{$d^4_{18}=(0)\vee(0)\vee(0)\vee(0)$};\\[5pt] 
        ${\ }^\partial\Gamma_{d^4_{18}}(z) = 16.$}
        \\
        \hline
        \end{tabular}
    }
\end{table}

\section{Main result} \label{s:main}
\begin{theorem}The partial-dual genus polynomial as a function on chord diagrams satisfies the 4-term relation.
\end{theorem}
\begin{proof}Each of the four chord diagrams of \eqref{eq:4T} involves two chords. One of them is stable (the same in all four diagrams) connecting the bottom arc with upper right arc; we call it the second chord. The other one, with one end at the upper left arc is changing. Its second end goes around the endpoints of the stable chord. We call it the first chord.
We are going to make all possible partial duals relative to these two chords; the partial duals for the chords not indicated in \eqref{eq:4T} are the same for all four diagrams. For example, the result of partial duality of the third diagram of \eqref{eq:4T} relative to the first chord will be denoted $G_{3;10}$ and the fourth diagram of \eqref{eq:4T} relative to both chords will be denoted $G_{4;11}$, etc. Alltogether we will get 16 ribbon graphs which will be partitioned into 8 pairs (indicated by colored frames), so that the ribbon graphs in each pair are obtained from one to another by a sliding operation from Sec.\ref{ss:rg-cd}. Thus, as abstract surfaces, they have the same genus, and therefore cancel each other in the partial-dual genus polynomial.
$$\begin{array}{rcrrrr}
+\risS{-15}{4T1}{\put(25,-6){$G_1$}}{35}{20}{45} & 
   \risS{0}{totor}{}{20}{0}{0}&
  +\risS{-18}{G-1-00}{\put(15,-10){$G_{1;00}$}}{40}{0}{0}&
	+\risS{-18}{G-1-10}{\put(15,-10){$G_{1;10}$}}{40}{0}{0}& 
	+\risS{-18}{G-1-01}{\put(15,-10){$G_{1;01}$}}{40}{0}{0}& 
	+\risS{-18}{G-1-11}{\put(15,-10){$G_{1;11}$}}{40}{0}{0}\\ 
	
-\risS{-15}{4T2}{\put(25,-6){$G_2$}}{35}{20}{45} & 
   \risS{0}{totor}{}{20}{0}{0}&
  -\risS{-18}{G-2-00}{\put(15,-10){$G_{2;00}$}}{40}{0}{0}&
	-\risS{-18}{G-2-10}{\put(15,-10){$G_{2;10}$}}{40}{0}{0}& 
	-\risS{-18}{G-2-01}{\put(15,-10){$G_{2;01}$}}{40}{0}{0}& 
	-\risS{-18}{G-2-11}{\put(15,-10){$G_{2;11}$}}{40}{0}{0}\\ 
	
+\risS{-15}{4T4}{\put(25,-6){$G_3$}}{35}{20}{45} & 
   \risS{0}{totor}{}{20}{0}{0}&
  +\risS{-18}{G-3-00}{\put(15,-10){$G_{3;00}$}}{40}{0}{0}&	
	+\risS{-18}{G-3-10}{\put(15,-10){$G_{3;10}$}}{40}{0}{0}&  
	+\risS{-18}{G-3-01}{\put(15,-10){$G_{3;01}$}}{40}{0}{0}&  
  +\risS{-18}{G-3-11}{\put(15,-10){$G_{3;11}$}}{40}{0}{0}\\ 

-\risS{-15}{4T3}{\put(25,-6){$G_4$}}{35}{20}{30} & 
   \risS{0}{totor}{}{20}{0}{0}&
  -\risS{-18}{G-4-00}{\put(15,-10){$G_{4;00}$}}{40}{0}{0}&
	-\risS{-18}{G-4-10}{\put(15,-10){$G_{4;10}$}}{40}{0}{0}&  
	-\risS{-18}{G-4-01}{\put(15,-10){$G_{4;01}$}}{40}{0}{0}&  
	-\risS{-18}{G-4-11}{\put(15,-10){$G_{4;11}$}}{40}{0}{0}\\ 
\end{array}
$$
The ribbon graphs of the first column form two pairs. The ribbon graphs in these pairs can be obtained from one another by sliding the first chord along the second one. Therefore, their contributions to the partial-dual genus polynomial cancel each other. Here are the sliding operations for other pairs.

$$G_{1;10} =\ \risS{-16}{G-1-10}{\put(-90,35){$(G_{1;10},G_{2;10}):$}
                             }{40}{25}{20} 
	\ =\ \risS{-10}{G-1-10a}{}{35}{0}{0}
	\quad\risS{0}{totor}{\put(1,10){\small slide}}{20}{0}{0}\quad
	  \risS{-13}{G-1-10b}{}{35}{0}{0}\ =\ 
		\risS{-16}{G-2-10}{}{40}{0}{0} 
	\ = G_{2;10}
$$
The slidings in pairs $(G_{3;10},G_{4;10})$, $(G_{1;01},G_{2;01})$, and
$(G_{3;01},G_{4;01})$ are very similar. The remaining two pairs of the last column work as follows.
$$G_{1;11} =\ \risS{-16}{G-1-11}{\put(-55,48){$(G_{1;11},G_{4;11}):$}
                             }{40}{40}{20} 
	\ =\ \risS{-12}{G-1-11a}{}{35}{0}{0}
	\quad\risS{0}{totor}{\put(1,10){\small slide}}{20}{0}{0}\quad
	  \risS{-12}{G-1-11b}{}{35}{0}{0}
	\quad\risS{0}{totor}{\put(1,10){\small slide}}{20}{0}{0}\quad
	  \risS{-12}{G-1-11c}{}{35}{0}{0}\ =\ 
		\risS{-16}{G-4-11}{}{40}{0}{0} 
	\ = G_{4;11}
$$
$$G_{2;11} =\ \risS{-16}{G-2-11}{\put(-55,48){$(G_{2;11},G_{3;11}):$}
                             }{40}{40}{20} 
	\ =\ \risS{-12}{G-2-11a}{}{35}{0}{0}
	\quad\risS{0}{totor}{\put(1,10){\small slide}}{20}{0}{0}\quad
	  \risS{-12}{G-2-11b}{}{35}{0}{0}
	\quad\risS{0}{totor}{\put(1,10){\small slide}}{20}{0}{0}\quad
	  \risS{-12}{G-2-11c}{}{35}{0}{0}\ =\ 
		\risS{-16}{G-3-11}{}{40}{0}{0} 
	\ = G_{3;11}
$$
\end{proof}

\section{Open problems} \label{s:op}
\begin{enumerate}
\item All known constructive examples of weight systems come from Lie algebras (see for example \cite{CDM}). They cover all quantum group knot invariants. On the other hand, there is a difficult theorem of P.~Vogel \cite{vog} claiming the existence of a Vassiliev invariant not coming from Lie algebras. But this is only an existence theorem, no concrete construction of such a weight system is known. Does the partial-dual genus polynomial weight system also come from Lie algebras?
If so, it would be interesting to explicitly express it in terms of Lie algebra weight systems. If not, then it will give a new constructive proof of Vogel's theorem.
\footnote{In December 2023 Iain Moffatt informed me that he proved that partial-dual genus polynomial weight system comes from the Lie algebra
$\mathfrak{gl}_N$ weight system.}

\item The primitive space of the bialgebra of chord diagrams can be describe it terms of {\it Jacobi diagrams} \cite{CDM}. It would be interesting to find a combinatorial construction of the partial-dual genus polynomial weight system on Jacobi diagrams. Possibly this would involve a generalization on the concept of partial duality to hypermaps from \cite{CVT22} and its further generalization to the partial duality of Jacobi diagrams.

\item According to \cite{CL}, Vassiliev invariants not distinguishing mutant knots are associated to weight systems which depend only on the {\it intersection graph} of chord diagrams. This means that the values of a weight system on two chord diagrams with the same intersection graphs are equal. Q.~Yan and X.~Jin \cite{YaJi22s} proved this for the intersection graphs of chord diagrams. Their main tool was a recurrent formula given in Theorem 5.2. In another paper \cite{YaJi22dm} (see also \cite{Yus22}) they
also extended the partial-dual genus polynomial to binary delta-matroids. Consequently, it can be considered as a new polynomial for all ({\it framed} in the  terminology of \cite{Lan-J}) graphs. It is interesting whether the recurrent formula from the Theorem 5.2. is also valid for all such graphs. Are there any relation of this polynomial with other graph polynomials?
%
%
Does it satisfy the four-term relation on graphs in sense of 
\cite{Lan,Lan-J}? If so, then it would provide a polynomial on the bialgebra of framed graphs. Its primitive space was recently described in
\cite{Kar}. It would be interesting to construct this polynomial on the primitive elements from \cite{Kar}.

\item If we take the average of a weight system over all graphs, we can get a solution of some interesting PDEs of mathematical physics \cite{CKL}. Would it be similar for the average of the partial-dual genus polynomial over all ribbon graphs?

\item It would be interesting to investigate probabilistic properties of the partial-dual genus polynomial similar to those of the genus of ribbon graphs in \cite{CP13,CP16}.
\end{enumerate}

\small{
    
}

\EditInfo{August 8, 2023}{December 19, 2023}{Jacob Mostovoy}
\end{document}